\documentclass[12pt]{amsart}
\numberwithin{equation}{section}
\usepackage{amsmath, amssymb}
\usepackage{hyperref}
\newtheorem{theorem}{Theorem}[section]
\newtheorem{corollary}[theorem]{Corollary}

\numberwithin{equation}{section}
\begin{document}

\title [ ]
{Hadamard product of convex functions and Jackson operator }
\author[ K. Piejko, J. Sok\'{o}\l, K. Tr\c{a}bka-Wi\c{e}c\a{l}aw]
       { Krzysztof Piejko$^1$, Janusz Sok\'{o}\l$^{2,*}$, Katarzyna
       Tr\c{a}bka-Wi\c{e}c\a{l}aw$^3$}

\address{$^1$ Rzesz\'{o}w University of Technology,
         Faculty of Mathematics and Applied Physics,
         Al. Powsta\'{n}c\'{o}w Warszawy 12, 35-959 Rzesz\'{o}w,
         Poland}
         \email{piejko@prz.edu.pl}
\address{$^{2}$ University of Rzesz\'{o}w,
                  ul. Prof. Pigonia 1, 35-310 Rzesz\'{o}w, Poland}
         \email{jsokol@ur.edu.pl}
\address{$^3$ Lublin University of Technology,
         Mechanical Engineering Faculty,
         ul. Nadbystrzycka 36, 20-618 Lublin, Poland}
         \email{k.trabka@pollub.pl}

\date{}

\begin{abstract}
In this paper we consider some properties of Jackson's difference
operator for convex univalent functions in $|z|<1$ with complex
parameter $q$ as a Hadamard product of two power series. Jackson in
1908 introduced for a real $q$, $q\in[0,1)$, the difference operator
\mbox{${\rm d}_qf(z)$} for an analytic function $f$ in the unit disc
$|z|<1$ in the complex plane. Thanks to this operator, many
mathematicians have extended the theory of functions in $q$-theory.
The $q$-theory has found many applications in theory of
hypergeometric series, special functions, combinatorics, number
theory, fluid mechanics, quantum mechanics and physics.
\end{abstract}


\keywords{analytic functions;  convex functions; starlike functions; difference operator;
$q$-calculus; $q$-starlike; $q$-derivative\\
$^*$ corresponding author: University of Rzesz\'{o}w,
         Faculty of Mathematics and Natural Sciences,
         ul. Prof. Pigonia 1, 35-310 Rzesz\'{o}w, Poland,
         jsokol@ur.edu.pl}


\maketitle


\section{ Introduction}

Let ${\mathbb C}$ denote the open complex plane and $\mathbb D$ be the unit disc,
${{\mathbb D}}:=\{z\in {\mathbb{C}}:\ |z|<1\}$. In this paper we will use
the following well-known definitions and notations. Let ${\mathcal H}$ denote the
class of analytic functions in $\mathbb D$. Let ${\mathcal A}$ be
the subclass of ${\mathcal H}$ consisting of functions normalized by
$f(0)=0$, $f'(0)=1$, i.e.,
\begin{equation}\label{f}
f(z)=z+\sum_{n=2}^{\infty} a_nz^n, \quad z\in\mathbb D
\end{equation}
and let $\mathcal {S}\subset\mathcal{A}$ be the class of functions
which are univalent in ${\mathbb D}$. Let $\mathcal S^*(\alpha)$,
$\alpha\in[0,1)$, denote the class of functions $f\in\mathcal A$ such
that
\begin{equation*}
    \mathfrak{Re}\left\{\frac{zf'(z)}{f(z)}\right\}>\alpha, \quad z\in\mathbb
    D.
\end{equation*}
Functions belonging to the class $\mathcal S^*(\alpha)$ are called
{\emph{starlike univalent functions of order}} $\alpha$, while $\mathcal
S^*=\mathcal S^*(0)$ is called the class of {\emph{starlike univalent
functions (with respect to the origin)}}. This is because if
$f\in\mathcal S^*$, then $f(\mathbb D)$ is a starlike set with
respect to the origin. Recall that a set $E\subset {\mathbb{C}}$ is
said to be starlike with respect to a point $w_0\in E$, if and only
if, the linear segment joining $w_0$ to any other point $w\in E$
lies entirely in $E$. In \cite{Rbs} Robertson extended $\mathcal
S^*$ to $\mathcal S^*(\alpha)$, but the geometric interpretation of
$f(\mathbb D)$ in this class is still an open problem. Let $\mathcal
K$ denote the class of functions $f\in\mathcal A$ such that
\begin{equation*}
    \mathfrak{Re}\left\{1+\frac{zf''(z)}{f'(z)}\right\}>0, \quad z\in\mathbb
    D.
\end{equation*}
Functions belonging to the class $\mathcal K$ are called
{\emph{convex univalent functions}} because
\begin{equation*}
    f\in\mathcal K\Rightarrow [f\in\mathcal S\quad {\rm and}\quad
f(\mathbb D)~{\rm is~a~convex~set}].
\end{equation*}
It is noted that $f\in\mathcal K\Leftrightarrow zf'(z)\in\mathcal S^*$, $z\in \mathbb D$.
For $f(z)=a_0+a_1z+a_2z^2+\ldots$ and $g(z)=b_0+b_1z+b_2z^2+\ldots$
the Hadamard product (or convolution) is defined by $ (f\ast
g)(z)=a_0b_0+a_1b_1z+a_2b_2z^2+\ldots $. In \cite{RusSSl73} it was
proved the famous result that for all $f,g\in \mathcal K$ and $h\in
\mathcal S^*(\alpha)$ we have
\begin{equation}\label{RusSS}
    f\ast g\in \mathcal K\quad {\rm{and}}\quad f\ast h\in \mathcal
    S^*(\alpha).
\end{equation}
In this paper we replace in some  formulas the derivative
$f'$ with the derivative operator ${\rm d}_\zeta f$ which is described in the
next section. We compare some results for $f'$ with
analogous results for ${\rm d}_\zeta f$.

\section{ $\zeta$-derivative operator}

For given $\zeta$, $|\zeta|\leq 1$, we consider the function
$h_\zeta$ of the form
\begin{equation*}
    h_\zeta(z)=\frac{z}{(1-\zeta z)(1-z)}
    = \sum_{n=1}^\infty \frac{1-\zeta^n}{1-\zeta}\ z^n, \quad z\in\mathbb
    D.
\end{equation*}
Then
\begin{equation}\label{hzeta}
    \mathfrak{Re}\left\{\frac{zh'_\zeta(z)}{h_\zeta(z)}\right\}
    =\mathfrak{Re}\left\{1+\frac{z}{1-z}+\frac{\zeta z}{1-\zeta z}\right\}
    >1-\frac{1}{2}-\frac{|\zeta|}{1+|\zeta|}
    =\frac{1-|\zeta|}{2(1+|\zeta|)}\geq 0, \quad z\in\mathbb
    D,
\end{equation}
so $h_\zeta$ is starlike for all complex $\zeta$, $|\zeta|\leq1$. It
is easy to check that if $\zeta =1$, then the function $h_{\zeta}$
becomes the well-known Koebe function
\begin{equation*}
    h_1(z)=\frac{z}{(1-z)^2}=\sum_{n=1}^\infty nz^n,\quad z\in\mathbb D.
\end{equation*}
For each $f\in\mathcal A$ we can express its derivative in terms of
convolution and the Koebe function as
\begin{equation}\label{a1}
    f'(z)= \frac{1}{z} \left\{f(z)*h_1(z)\right\}
    =\frac{1}{z} \left\{f(z)*\frac{z}{(1-z)^2}\right\}, \quad z\in\mathbb
    D.
\end{equation}
 In \cite{PKJS10} the
following generalization of \eqref{a1} for $\zeta\in\mathbb C$,
$|\zeta|\leq 1$, was defined
\begin{equation}\label{a2}
     {\rm d}_\zeta f(z)= \frac{1}{z} \left\{f(z)*h_\zeta(z)\right\}
     =\frac{1}{z}\left\{f(z)*\frac{z}{(1-\zeta z)(1-z)}\right\}.
\end{equation}
For $\zeta=1$ convolution operator \eqref{a2} becomes the
derivative $f'$. For $f$ given by \eqref{f} we have:
\begin{eqnarray}\label{a3}
  {\rm d}_\zeta f(z) &=& \frac{1}{z} \left\{f(z)*\frac{z}{(1-\zeta z)(1-z)} \right\} \\
   &=& \frac{1}{z} \left\{ \left(z+\sum_{n=2}^\infty a_nz^n \right) * \left( \sum_{n=1}^\infty \frac{1-\zeta^n}{1-\zeta}\ z^n \right) \right\} \nonumber \\
   &=& 1+\sum_{n=2}^\infty
   \frac{1-\zeta^n}{1-\zeta}\ a_nz^{n-1} =
   1+\sum_{n=2}^\infty[n]_{\zeta} a_nz^{n-1},  \nonumber
\end{eqnarray}
where
\begin{equation*}
    [n]_{\zeta}=\frac{1-\zeta^n}{1-\zeta}, \quad n=2,3,\ldots \ .
\end{equation*}

A natural question is what kind of results on $f'$ hold also for
$\zeta$-derivative operator ${\rm d}_\zeta f$. It is known that if
$f$  is analytic in $\mathbb D$ and for some real $\alpha$ we have
\begin{equation*}
    \mathfrak{Re} \{\exp(i\alpha)f'(z)\} > 0,\quad z\in \mathbb D,
\end{equation*}
then $f$ is univalent in $\mathbb D$. The unit disc may be replaced
with any convex domain too. This is the well-known
Noshiro-Warschawski univalence condition, \cite{Nosh,Warschawski}.
Such condition does not hold for ${\rm d}_\zeta f$ instead of $f'$,
apart the case $\zeta=1$ when ${\rm d}_\zeta f=f'$. It is known that
$f(z)=z+az^2$ is univalent in $\mathbb D$, if and only if, $|a|\leq
1/2$. Therefore, for given $\zeta$, $|\zeta|<1$, the function
\begin{equation*}
    f(z)=z+\frac{1}{1+\zeta}z^2
\end{equation*}
is not univalent in $\mathbb D$, although
\begin{equation*}
    \mathfrak{Re}\left\{ {\rm d}_\zeta f(z) \right\}=\mathfrak{Re}\left\{1+z \right\}
    >0,\quad z\in \mathbb D.
\end{equation*}

Jackson in \cite{Jakson1908,Jakson1910} introduced and studied the
$q$-derivative, $0\leq q<1$, as
\begin{equation}\label{J1}
    {\rm d}_qf(z)=\frac{f(qz)-f(z)}{(q-1)z}, \quad z\neq 0 \quad \text{and} \quad {\rm d}_qf(0)=f'(0).
\end{equation}
Thus, from \eqref{J1} for a function $f$
given by \eqref{f} we have
\begin{eqnarray}\label{J2}
  z{\rm d}_qf(z) &=&  z+\sum_{n=2}^\infty[n]_{q} a_nz^n,
\end{eqnarray}
where
\begin{equation*}
    [n]_{q}=\frac{1-q^n}{1-q}, \quad n=2,3,\ldots \ .
\end{equation*}
It is clear that for $\zeta=q$, $0\leq q<1$, the $\zeta$-derivative
operator  given in \eqref{a3}, becomes the Jackson $q$-derivative of
$f$ defined in \eqref{J1}. Therefore,  we can also look at Jackson's
$q$-derivative  as a special case of convolution operator \eqref{a2}.
In fact, it follows that the considered convolution operator is an
analytic extension of the Jackson derivative from the segment to the disk.

\noindent {\bf Definition.}\cite{PKJS10,PKJS11} Let $f\in \mathcal
A$. For a given $\zeta$, $|\zeta|\leq 1$, we say that $f$ is in the
class $\mathcal R(\zeta,\alpha)$, $0\leq \alpha<1$, if
\begin{equation}\label{DEF}
    \mathfrak{Re} \left\{\frac{f'(z)}{ {\rm d}_{\zeta}f(z)}\right\} >\alpha,\quad z\in\mathbb D,
\end{equation}
where the operator ${\rm d}_{\zeta}$ is defined in \eqref{a2}.

\noindent {\bf Remark 1.} It is easy to see that $\mathcal
R(1,\alpha)=\mathcal A$.

\noindent {\bf Remark 2.} For the function $h(z)=z/(1-z)$, we have
\begin{equation*}
       \mathfrak{Re} \left\{\frac{h'(z)}{ {\rm d}_{\zeta}h(z)}\right\}
       = \mathfrak{Re} \left\{\frac{1-\zeta z}{1-z}\right\}
       >\frac{1+|\zeta|}{2},\quad z\in\mathbb D,
\end{equation*}
so for a given $\zeta$, $|\zeta|\leq 1$, the function $h\in\mathcal
R(\zeta,(1+|\zeta|)/2)$.

{\bf Remark 3.} For $f$ given by \eqref{f} condition \eqref{DEF} becomes
\begin{equation}\label{DEF2}
        \mathfrak {Re}\left\{ \frac{1+\sum_{n=2}^\infty na_nz^{n-1} }{1+\sum_{n=2}^\infty[n]_{\zeta} a_nz^{n-1} }\right\}>\alpha, \quad z\in\mathbb D.
\end{equation}
\section{Main results}

\begin{theorem}\label{t22}
If $f$ is in the class $\mathcal K$ of convex univalent functions,
then
\begin{equation}\label{1t22}
    \mathfrak{Re} \left\{ \frac{\zeta}{1-\zeta}\left(\frac{f'(z)}{\zeta {\rm d}_{\zeta}f(z)}-1\right) \right\} >\frac{1}{2}
\end{equation}
for all $z\in\mathbb D$ and $\zeta$, $|\zeta|\leq1$. The result is
the best possible.
\end{theorem}

\begin{proof}

First we prove that for each $f\in\mathcal K$ the following
inequality
\begin{equation}\label{3t22}
    \mathfrak{Re} \left\{\left(\frac{e^{ia}}{1-e^{ia}}-\frac{e^{ib}}{1-e^{ib}}\right)
    \frac{f(e^{ib}z)-f(z)}{f(z)-f(e^{ia}z)} \right\}>0
\end{equation}
holds for all $z\in\mathbb D$ and for all real $a,b$ such that
$0<b<a<\pi$. We have
\begin{eqnarray}\label{4t22}
          \frac{e^{ia}}{1-e^{ia}}-\frac{e^{ib}}{1-e^{ib}} &=& \left(-
      \frac{1}{2}+\frac{i}{2}\cot\frac{a}{2}\right)-\left(-
      \frac{1}{2}+\frac{i}{2}\cot\frac{b}{2}\right)\nonumber \\
       &=& \frac{i}{2}\left(\cot\frac{a}{2}- \cot\frac{b}{2}\right).
\end{eqnarray}
In this way we obtain
\begin{eqnarray*}
    \mathfrak{Re}\left\{\frac{e^{ia}}{1-e^{ia}}-\frac{e^{ib}}{1-e^{ib}}\right\}=0\quad{\rm
    and}\quad
    \mathfrak{Im}\left\{\frac{e^{ia}}{1-e^{ia}}-\frac{e^{ib}}{1-e^{ib}}\right\}<0.
\end{eqnarray*}
Therefore, to calculate the real part in \eqref{3t22}, we need the
imaginary part of the second factor in \eqref{3t22}. The curve
$f(w=|z|)$ is convex so as $x$ runs the interval $[b,a]$ then the
vector $f(e^{ix}z)-f(z)$ turns around $f(z)$ in the counter-clock
vise direction no more than $\pi$. In the other words,
$\arg\{f(e^{ix}z)-f(z)\}$ increases no more than $\pi$. This gives
\begin{equation}\label{5t22}
    \arg\frac{f(e^{ib}z)-f(z)}{f(e^{ia}z)-f(z)}\in(-\pi,0).
\end{equation}
So we have
\begin{equation*}
    \mathfrak{Im} \left\{\frac{f(e^{ib}z)-f(z)}{f(e^{ia}z)-f(z)}
    \right\}<0.
\end{equation*}
From this we have
\begin{align}\label{55t22}
    \mathfrak{Re} \left\{\left(\frac{e^{ia}}{1-e^{ia}}-\frac{e^{ib}}{1-e^{ib}}\right)
    \frac{f(e^{ib}z)-f(z)}{f(z)-f(e^{ia}z)} \right\}\nonumber
   &= -\mathfrak{Im} \left\{\frac{e^{ia}}{1-e^{ia}}-\frac{e^{ib}}{1-e^{ib}}\right\}
    \mathfrak{Im} \left\{\frac{f(e^{ib}z)-f(z)}{f(z)-f(e^{ia}z)} \right\}\nonumber \\
   &=\mathfrak{Im}
    \left\{\frac{e^{ia}}{1-e^{ia}}-\frac{e^{ib}}{1-e^{ib}}\right\}
    \mathfrak{Im} \left\{
    \frac{f(e^{ib}z)-f(z)}{f(e^{ia}z)-f(z)}\right\}\\
   &>0\nonumber
\end{align}
for all $z\in\mathbb D$ and for all real $a,b$ such that
$0<b<a<\pi$. This proves \eqref{3t22}.  Inequality \eqref{3t22} may
be extended to the case for $-\pi<a<b<0$ (in this case both
imaginary parts in \eqref{55t22} becomes positive). Therefore
\eqref{3t22} holds for all $z\in\mathbb D$ and for all real $a,b$,
$0<|b|<|a|<\pi$.

Inequality \eqref{3t22} may be written in the following equivalent
form
\begin{equation}\label{6t22}
    \mathfrak{Re} \left\{
    \frac{f(e^{ib}z)-f(z)}{e^{ib}z-z}\frac{z(e^{ib}-e^{ia})}{(1-e^{ia})(f(z)-f(e^{ia}z))}
    \right\}>0.
\end{equation}
If $b\rightarrow 0$, then \eqref{6t22} becomes
\begin{equation}\label{7t22}
    \mathfrak{Re} \left\{
    f'(z)\frac{z}{f(z)-f(e^{ia}z)}
    \right\}>0
\end{equation}
for all $z\in\mathbb D$ and for all real $-\pi<a<\pi$, $a\neq 0$. It
is clear that \eqref{7t22} holds for $a=\pi$. We shall consider the
function
\begin{equation*}
    h(\zeta)=\frac{f'(z)}{(1-\zeta) {\rm d}_{\zeta}f(z)}-\frac{1+\zeta}{2(1-\zeta)},
\end{equation*}
with  $f\in\mathcal K$. If $f$ is of the form \eqref{f}, then we see
that
\begin{eqnarray*}
    h(\zeta)&=&\frac{1}{1-\zeta}\left\{\frac{f'(z)}{ {\rm
    d}_{\zeta}f(z)}-1\right\}+\frac{1}{2}\\
    &=&\frac{1}{1-\zeta}\frac{\sum_{n=2}^\infty\left(n-(\zeta^{n-1}+\cdots+\zeta+1)\right)a_nz^{n-1}}
    {\sum_{n=1}^\infty(\zeta^{n-1}+\cdots+\zeta+1)a_nz^{n-1}}+\frac{1}{2}\\
    &=&\frac{\sum_{n=2}^\infty\left(\zeta^{n-2}+2\zeta^{n-3}+3\zeta^{n-4}+\cdots+(n-2)\zeta+n-1\right)a_nz^{n-1}}
    {\sum_{n=1}^\infty(1+\zeta+\cdots+\zeta^{n-1})a_nz^{n-1}}+\frac{1}{2},
    \end{eqnarray*}
so the function $h(\zeta)$ is analytic in $|\zeta|\leq 1$.
Furthermore, we have
\begin{eqnarray}\label{8t22}
    \mathfrak{Re}\left\{h(1)\right\}
    &=&\mathfrak{Re}\left\{\frac{\sum_{n=2}^\infty\left(n(n-1)/2\right)a_nz^{n-1}}{\sum_{n=1}^\infty na_nz^{n-1}}+\frac{1}{2}\right\}\nonumber\\
    &=&\mathfrak{Re}\left\{\frac{\sum_{n=2}^\infty n(n-1) a_nz^{n-1}}{\sum_{n=1}^\infty
    2na_nz^{n-1}}+\frac{1}{2}\right\}\nonumber\\
    &=&\mathfrak{Re}\left\{\frac{zf''(z)}{2f'(z)}+\frac{1}{2}\right\}\nonumber\\
    &>&0,
    \end{eqnarray}
because $\mathfrak{Re}\{zf''(z)/2f'(z)\}>-1/2$ for all $f\in\mathcal
K$. We have also
\begin{equation*}
    \mathfrak{Re}\{h(0)\}=\mathfrak{Re}\left\{\frac{-zf'(z)}{f(0)-f(z)}-\frac{1}{2}\right\}
    =\mathfrak{Re}\left\{\frac{zf'(z)}{f(z)}-\frac{1}{2}\right\}>0
\end{equation*}
because convex $f$ is starlike of order $1/2$ too. But
$\mathfrak{Re}\{h(|\zeta|<r)\}$, $r\in(0,1)$, attains its minimum on
$|\zeta|=r$. On the unit circle $|\zeta|=1$, we have by \eqref{8t22}
that $\mathfrak{Re}\left\{h(1)\right\}>0$ and for $\zeta=e^{ia}$,
$\zeta\neq1$, we have by \eqref{7t22}
\begin{equation*}
    \mathfrak{Re}\{h(e^{ia})\}=
    \mathfrak{Re} \left\{
    \frac{zf'(z)}{f(z)-f(e^{ia} z)}
    -\frac{1+e^{ia}}{2(1-e^{ia})}\right\}=\mathfrak{Re} \left\{
    \frac{zf'(z)}{f(z)-f(e^{ia} z)}\right\}>0.
\end{equation*}
This shows that $\mathfrak{Re}\left\{h(\zeta)\right\}>0$ on
$|\zeta|=1$. Therefore, we finally obtain
\begin{equation}\label{9t22}
    \mathfrak{Re}\{h(\zeta)\}>0
\end{equation}
for all $z\in\mathbb D$ and $|\zeta|\leq 1$. This and equality
\begin{equation*}
    h(\zeta)=\frac{f'(z)}{ (1-\zeta){\rm
    d}_{\zeta}f(z)}-\frac{\zeta}{1-\zeta}-\frac{1}{2}
\end{equation*}
give \eqref{1t22}. If we take the convex function $g(z)=z/(1-z)$,
then
\begin{equation*}
    \frac{\zeta}{1-\zeta}\left(\frac{g'(z)}{\zeta {\rm
    d}_{\zeta}g(z)}-1\right)=\frac{1}{1-z}
\end{equation*}
maps $\mathbb D$ onto the half-plane $\mathfrak{Re}\{w\}>1/2$, so
\eqref{1t22} is the best possible in the class $\mathcal K$.
\end{proof}

Note that \eqref{1t22} does not hold in the class of starlike
functions. The function $f(z)=z+z^2/2$ is in $\mathcal S^*$ but
\eqref{1t22} becomes
\begin{equation*}
    \mathfrak{Re} \left\{ \frac{\zeta}{1-\zeta}\left(\frac{1+z}{\zeta (1+(1+\zeta)z/2)}-1\right) \right\} >\frac{1}{2}
\end{equation*}
which evidently does not hold for all $z\in\mathbb D$ and $\zeta$,
$|\zeta|\leq1$.

 A simple calculation yields another form of
\eqref{1t22} in the following corollary.

\begin{corollary}\label{c50}

If $f$ is in the class $\mathcal K$ of convex univalent functions,
then
\begin{equation}\label{1c50}
    \mathfrak{Re} \left\{ \frac{f'(z)}{(1-\zeta) {\rm d}_{\zeta}f(z)} \right\}
    >\mathfrak{Re} \left\{\frac{1+\zeta}{2(1-\zeta)}\right\}
    >0
\end{equation}
and
\begin{equation}\label{2c50}
    \mathfrak{Re} \left\{ \frac{1}{1-\zeta}\left(\frac{f'(z)}{{\rm d}_{\zeta}f(z)}-1\right) \right\} >-\frac{1}{2}
\end{equation}
for all $z\in\mathbb D$ and $\zeta$ such that $|\zeta|\leq1$. The
results are the best possible.
\end{corollary}
In \cite[Lemma 2]{PKJS11} it was proved that for $f\in\mathcal K$
\begin{equation}\label{3c50}
    \mathfrak{Re} \left\{ \frac{f'(z)}{(1-\zeta) {\rm d}_{\zeta}f(z)} \right\}
    >0
\end{equation}
for all $z,\zeta\in\mathbb D$. Therefore, \eqref{1c50} improves
earlier result \eqref{3c50}.

\begin{corollary}\label{c550}

If $f$ is in the class $\mathcal K$ of convex univalent functions,
then
\begin{equation}\label{1c550}
    \mathfrak{Re} \left\{ \frac{f'(z)}{(1-\zeta) {\rm d}_{\zeta}f(z)} \right\}
    >\frac{1-|\zeta|^2}{2(1-2\mathfrak{Re}\zeta+|\zeta|^2)}
    =\frac{1-|\zeta|^2}{2|1-\zeta|^2}.
   \end{equation}
The result is the best possible.
\end{corollary}
\begin{proof}
Inequality \eqref{1c550} follows from \eqref{1c50} and from the
known fact that
\begin{equation*}
    \mathfrak{Re} \left\{\frac{1+\zeta}{2(1-\zeta)}\right\}=\frac{1-|\zeta|^2}{2(1-2\mathfrak{Re}\zeta+|\zeta|^2)}
\end{equation*}
for all $\zeta\in\mathbb D$. For convex function $h(z)=z/(1-z)$, we
have
\begin{equation*}
    \left.\mathfrak{Re} \left\{ \frac{h'(z)}{(1-\zeta) {\rm d}_{\zeta}h(z)}
    \right\}\right|_{z=-1}=\mathfrak{Re} \left\{ \frac{1+\zeta}{2(1-\zeta)}
    \right\}
    =\frac{1-|\zeta|^2}{2(1-2\mathfrak{Re}\zeta+|\zeta|^2)}.
   \end{equation*}
This shows that \eqref{1c550} is the best possible.
\end{proof}
If we take a real $q$ instead of $\zeta$ in \eqref{1c50}, then we
immediately obtain the following corollary.


\begin{corollary}\label{c51}

If $f$ is in the class $\mathcal K$ of convex univalent functions
and $q\in[0,1)$, then
\begin{equation}\label{1c51}
    \mathfrak{Re} \left\{\frac{f'(z)}{ {\rm d}_{q}f(z)}\right\} >\frac{1+q}{2}
\end{equation}
for all $z\in\mathbb D$, Equivalently, $f\in\mathcal K$ implies
$f\in\mathcal R(q,(1+q)/2)$. The result is the best possible.
\end{corollary}

\noindent {\bf Open problem.} Is it true that if $f$ is in the class
$\mathcal K$ of convex univalent functions, then
\begin{equation}\label{conjecture}
    \mathfrak{Re} \left\{ \frac{f'(z)}{ {\rm d}_{\zeta}f(z)} \right\}
    >\frac{1+|\zeta|}{2}\ ?
\end{equation}
For $q=0$ inequality \eqref{1c51} becomes
\begin{equation*}
    \mathfrak{Re} \left\{\frac{zf'(z)}{ f(z)}\right\} >\frac{1}{2}
\end{equation*}
for all $z\in\mathbb D$, this is a well known result for convex
univalent function $f$.

\begin{corollary}\label{c52}

If $f$ is in the class $\mathcal K$ of convex univalent functions
and $q\in[0,1)$, then there exists an analytic function $p\in
\mathcal H$, $p(0)=1$, such that
\begin{equation}\label{1c52}
    f'(z)=\left(\frac{1-q}{2}p(z)+\frac{1+q}{2}\right) {\rm d}_{q}f(z),\quad z\in\mathbb D
\end{equation}
and
\begin{equation*}
    \mathfrak{Re} \left\{p(z) \right\} >0,\quad z\in\mathbb
    D.
\end{equation*}
\end{corollary}

\begin{proof}
From \eqref{1c51}, we have
\begin{equation*}
    \mathfrak{Re} \left\{\frac{\frac{2f'(z)}{ {\rm d}_{q}f(z)}-(1+q)}{1-q}\right\}
    >0,\quad z\in\mathbb
    D.
\end{equation*}
Therefore, the function in the brackets is an analytic function $p$
with positive real part in $\mathbb D$ with normalization $p(0)=1$.
This gives \eqref{1c52}.
\end{proof}



\begin{theorem}\label{t25}
If $f$ is in the class $\mathcal K$ of convex univalent functions,
then
\begin{equation}\label{1t25}
    \frac{1+qr}{1+r}\leq\mathfrak{Re} \left\{ \frac{f'(z)}{ {\rm d}_{q}f(z)} \right\}
    \leq \frac{1-qr}{1-r}
\end{equation}
and
\begin{equation}\label{3t25}
    \frac{1+qr}{1+r}\leq \left| \frac{f'(z)}{ {\rm d}_{q}f(z)}
    \right |
    \leq \frac{1-qr}{1-r}
\end{equation}
for all  $z, q$ such that $|z|=r<1$, $q\in[0,1)$. The result is
sharp.
\end{theorem}

\begin{proof}
From \eqref{1c51}, we have
\begin{equation*}
    \frac{f'(z)}{ {\rm d}_{q}f(z)}\prec \frac{1+(1-2\beta)z}{1-z},
    \quad \beta=\frac{1+q}{2}.
\end{equation*}
This gives
\begin{equation}\label{2t25}
    \frac{f'(z)}{ {\rm d}_{q}f(z)}\prec \frac{1-qz}{1-z}.
\end{equation}
It is known that
\begin{equation*}
    \frac{1+qr}{1+r}\leq\mathfrak{Re} \left\{ \frac{1-qz}{1-z} \right\}
    \leq \frac{1-qr}{1-r}
\end{equation*}
for all  $z, q$, $|z|=r<1$, $q\in[0,1)$. This shows \eqref{1t25}.
For convex function $h(z)=z/(1-z)$, we have
\begin{equation*}
    \frac{h'(z)}{ {\rm d}_{q}h(z)}= \frac{1-qz}{1-z}.
\end{equation*}
This shows the sharpness of \eqref{1t25}. The same argumentation
shows that subordination \eqref{2t25} implies \eqref{3t25}.
\end{proof}

\begin{theorem}\label{t26}
If $f$ is in the class $\mathcal K$ of convex univalent functions,
then
\begin{equation}\label{1t26}
        \mathfrak{Re}
        \left\{\frac{f(z)}{ \left[\log\frac{1}{1-z}\right]
        *z{\rm d}_{q}f(z)}\right\} >\frac{1+q}{2}
\end{equation}
for all  $z, q$ such that  $z\in\mathbb D$, $q\in[0,1)$.
\end{theorem}


\begin{proof}
It is known, see \cite[p.171]{Goodman-Book}, that if $N\in\mathcal
H$, $D\in\mathcal H$, $N(0)=D(0)=0$ and $D$ is starlike univalent in
$\mathbb D$, then
\begin{equation}\label{2t26}
    \mathfrak{Re} \left\{\frac{N'(z)}{D'(z)}\right\}>\alpha \quad
    \Rightarrow\quad\mathfrak{Re} \left\{\frac{N(z)}{D(z)}\right\}>\alpha
\end{equation}
for all  $z\in\mathbb D$. From \eqref{1c51} we have
\begin{equation}\label{3t26}
    \mathfrak{Re} \left\{\frac{f'(z)}{ {\rm d}_{q}f(z)}\right\}
    >\frac{1+q}{2}.
\end{equation}
On the other hand, we have
\begin{eqnarray*}
  {\rm d}_{q}f(z) &=& \sum_{n=1}^\infty[n]_{q} a_nz^{n-1}\\
  &=&\left\{\sum_{n=1}^\infty[n]_{q} a_nz^{n}
  *\sum_{n=1}^\infty\frac{1}{n}z^{n}\right\}'
   \\
   &=& \left\{z{\rm d}_{q}f(z)*\log\frac{1}{1-z}\right\}',
\end{eqnarray*}
where $z\mapsto\log\frac{1}{1-z}$ is convex and $z{\rm
d}_{q}f(z)=f(z)*h_q(z)$ is starlike by \eqref{RusSS} and
\eqref{hzeta}. Now, applying \eqref{2t26} and \eqref{3t26} give us
\eqref{1t26}.
\end{proof}
\section{Ethical Approval}
Not applicable.
\section{Competing interests}
The authors declare they have no financial interests.
The authors have no conflicts of interest to declare that are relevant to the content of this article.
\section{Authors' contributions}
The   authors declare they have equal contribution in the paper.
\section{Data availability}
Data sharing is not applicable to this article as no datasets were generated or analysed during the current study.
\section{Acknowledgements}

The third author is supported by the scientific project
FD-20/IM-5/126
funded by  Lublin University of Technology..

\end{document}